\documentclass[a4paper,twopage,reqno,12pt]{amsart}
\usepackage[top=30mm,right=30mm,bottom=30mm,left=30mm]{geometry}

\usepackage{graphicx}
\usepackage{amsmath}
\usepackage{amssymb}
\usepackage{amsfonts}
\usepackage{amsthm}
\usepackage{amstext}
\usepackage{amsbsy}
\usepackage{amsopn}
\usepackage{amscd}
\usepackage{enumerate}
\usepackage{color}
\usepackage[colorlinks=true,linkcolor=red,citecolor=blue]{hyperref}
\usepackage{multirow}
\usepackage{lipsum}
\usepackage{rotating}
\usepackage{float}
\usepackage{lscape}
\usepackage[]{algorithm2e}
\usepackage{tabularx}

\newtheorem{theorem}{Theorem}[section]
\newtheorem{corollary}[theorem]{Corollary}
\newtheorem{lemma}[theorem]{Lemma}

\theoremstyle{definition}

\newtheorem{remark}[theorem]{Remark}

\numberwithin{equation}{section}

\newcommand{\GL}{\mathrm{GL}}
\newcommand{\SL}{\mathrm{SL}}

\newcommand{\SU}{\mathrm{SU}}

\newcommand{\PSL}{\mathrm{PSL}}
\newcommand{\PSU}{\mathrm{PSU}}

\newcommand{\PGL}{\mathrm{PGL}}

\newcommand{\AGaL}{\mathrm{A}\Gamma \mathrm{L}}
\newcommand{\POm}{\mathrm{P} \Omega}

\newcommand{\Alt}{\mathrm{Alt}}
\newcommand{\Sym}{\mathrm{Sym}}
\newcommand{\A}{\mathrm{Alt}}

\newcommand{\D}{\mathrm{D}}
\renewcommand{\S}{\mathrm{S}}
\newcommand{\HS}{\mathrm{HS}}

\newcommand{\Aut}{\mathrm{Aut}}

\newcommand{\Out}{\mathrm{Out}}

\newcommand{\Zbb}{\mathbb{Z}}

\newcommand{\Dmc}{\mathcal{D}}
\newcommand{\Bmc}{\mathcal{B}}

\newcommand{\Pmc}{\mathcal{P}}

\newcommand{\Smc}{\mathcal{S}}
\newcommand{\Lmc}{\mathcal{L}}
\newcommand{\Umc}{\mathcal{U}}

\newcommand{\e}{\epsilon}

\renewcommand{\leq}{\leqslant}
\renewcommand{\geq}{\geqslant}

\newcommand{\imod}[1]{\allowbreak\mkern4mu({\operator@font mod}\,\,#1)}

\begin{document}
 \title[Flag-transitive automorphism groups of $2$-designs]{Almost simple groups as flag-transitive automorphism groups of $2$-designs with $\lambda=2$}

 \author[S.H. Alavi]{Seyed Hassan Alavi}%
 \address{Seyed Hassan Alavi, Department of Mathematics, Faculty of Science, Bu-Ali Sina University, Hamedan, Iran.
 }%
 \email{alavi.s.hassan@basu.ac.ir and  alavi.s.hassan@gmail.com}
 %

 \date{\today}%

\begin{abstract}
In this article, we study $2$-designs with $\lambda=2$ admitting a flag-transitive almost simple automorphism group with socle a finite simple exceptional group of Lie type, and we prove that such a $2$-design does not exist. In conclusion, we present a classification of $2$-designs with $\lambda=2$ admitting flag-transitive and point-primitive automorphism groups of almost simple type, which states that such a $2$-design belongs to an infinite family of $2$-designs with parameter set $((3^n-1)/2,3,2)$ and $X=\PSL_n(3)$ for some $n\geq 3$, or it is isomorphic to the $2$-design with parameter set $(6,3,2)$, $(7,4,2)$, $(10,4,2)$, $(10,4,2)$, $(11,5,2)$, $(28,7,2)$, $(28,3,2)$, $(36,6,2)$, $(126,6,2)$ or $(176,8,2)$. 
\end{abstract}

\maketitle

\section{Introduction}\label{sec:intro}

In $1990$, Buekenhout, Delandtsheer, Doyen, Kleidman, Liebeck and Saxl \cite{a:BDDKLS90} classified all flag-transitive linear spaces apart from those admitting a one-dimensional affine automorphism group. Since then, there have been numerous efforts to classify $2$-designs $\Dmc$ with parameters $(v,k,2)$ admitting flag-transitive automorphisms group $G$. In the case where $\Dmc$ is symmetric, in a series of papers,  Regueiro \cite{a:Regueiro-Exp,a:Regueiro-alt-spor,a:Regueiro-classical,a:Regueiro-classification,a:Regueiro-reduction} proved that either $(v,k) \in \{(7,4),(11,5),(16,6)\}$, or $G \leq \AGaL_{1}(q)$ for some odd prime power $q$. Liang and Zhou \cite{a:Zhou-lam2-nonsym} proved that, if $\Dmc$ is non-symmetric and $G$ is point-primitive, then $G$ is affine or almost simple. The case when $G$ is point-imprimitive has recently been treated by Devillers, Liang, Praeger and Xia \cite{a:DLPB-2021-lam2}, and they proved that if $G$ is flag-transitive and point-imprimitive, then $\Dmc$ is one of the two known symmetric $2$-$(16, 6, 2)$ designs, one with automorphism group $2^{4}\S_{6}$ and point-stabiliser $(\Zbb_{2}\times \Zbb_{8})(\S_{4}{\cdot}2)$ and the other with automorphism group $\S_{6}$ and point-stabiliser $\S_{4}{\cdot}2$, see also \cite[Examples 1.2]{a:Regueiro-reduction}. Therefore, by \cite[Theorem 1.1]{a:DLPB-2021-lam2}, the study of flag-transitive $2$-$(v,k,2)$ designs reduces to the case where the  automorphism group $G$ is point-primitive of affine or almost simple type. This paper is devoted to studying  $2$-$(v,k,\lambda)$ designs admitting flag-transitive automorphism groups of almost simple type. 

Regueiro \cite{a:Regueiro-Exp,a:Regueiro-alt-spor,a:Regueiro-classical,a:Regueiro-classification,a:Regueiro-reduction} classified all such examples when the design is symmetric (up to those admitting a one-dimensional affine automorphism group).  In the non-symmetric case, Liang and Zhou \cite{a:Zhou-lam2-nonsym-An,a:Zhou-lam2-nonsym} have dealt with the cases where the socle of $G$ is a sporadic simple group or an alternating group. Devillers, Liang, Praeger and Xia \cite{a:DLPB-2021-lam2} have studied the case where $G$ has socle $\PSL_{n}(q)$ with $n\geq 3$, and proved that either $\Dmc$ is the complement of the Fano plane, or $\Dmc$ is the unique $2$-design with parameters $(v,3,2)$, where $G$ has socle $\PSL_n(3)$ and $v=(3^n-1)/2$ for some $n\geq 3$, see \cite[Theorem 1.2]{a:DLPB-2021-lam2}. The case where $G$ has socle a finite classical group has been treated recently in \cite{a:ABDT-ASlam2}. In this paper, we finish the almost simple case, and prove that we have no example when the  socle is a finite simple exceptional group:

\begin{theorem}\label{thm:main}
    Let $\Dmc=(\Pmc,\Bmc)$ be a nontrivial $(v, k, 2)$ design, and let $G$ be a flag-transitive and point-primitive automorphism group of $\Dmc$ whose socle $X$ is a non-abelian finite simple group. Then $X$ cannot be a finite simple exceptional group of Lie type.
\end{theorem}

Theorem \ref{thm:main} together with the  main results in  \cite{a:ABDT-ASlam2,a:DLPB-2021-lam2,a:Zhou-lam2-nonsym,a:Zhou-lam2-nonsym-An,a:Montinaro-PSL2lam2} gives a classification of nontrivial $2$-designs with $\lambda=2$ admitting flag-transitive and point-primitive almost simple automorphism groups: 

\begin{corollary}\label{cor:main}
    Let $\Dmc$ be a nontrivial $2$-design with parameters $(v,k,2)$ and let $G$ be a flag-transitive and point-primitive almost simple automorphism group with socle $X$. Then  $(\Dmc,X)$ is as in one of the lines in {\rm Table~\ref{tbl:main}}, or $\Dmc$ is the $2$-design with parameters $((3^n-1)/2,3,2)$ and $X=\PSL_n(3)$ for some $n\geq 3$ constructed as in {\rm \cite[Theorem 1.2]{a:DLPB-2021-lam2}}.
\end{corollary}

\begin{table}
    \centering
    \small
    \caption{Some flag-transitive and point-primitive nontrivial $2$-design with $\lambda=2$.}\label{tbl:main}
    \begin{tabular}{clllllllll}
        \noalign{\smallskip}\hline\noalign{\smallskip}
        Line &
        $v$ &
        $b$ &
        $r$ &
        $k$ &
        $X$ &
        $X_{\alpha}$ &
        $\Aut(\Dmc)$ &
        References \\
        \noalign{\smallskip}\hline\noalign{\smallskip}
        $1$ &
        $6$ &
        $10$ &
        $5$ &
        $3$ &
        $\PSL_{2}(4)\cong \PSL_{2}(5)\cong \A_{5}$ &
        $\D_{8}$ &
        $\PSL_{2}(5)$ &
        \cite{a:ABD-Un-CP,a:ABD-Un-CP-cor} \\
        $2$ &
        $7$ &
        $7$ &
        $4$ &
        $4$ &
        $\PSL_{2}(7)\cong \PSL_{3}(2)$ &
        $\S_{4}$ &
        $\PSL_{2}(7)$ &
        \cite{a:ABD-PSL2,b:Handbook,a:Kantor-85-2-trans,a:Regueiro-classical}\\
        $3$ &
        $10$ &
        $15$ &
        $6$ &
        $4$ &
        $\PSL_{2}(4)\cong \PSL_{2}(5)\cong \A_{5}$&
        $\D_{12}$ &
        $\PSL_{2}(4):2$&
        \cite{a:Zhou-lam2-nonsym-An}
        \\%
        $4$ &
        $10$ &
        $15$ &
        $6$ &
        $4$ &
        $\PSL_{2}(9)\cong \A_{6}$&
        $3^{2}:4$ &
        $\PSL_{2}(9):2$&
        \cite{a:Zhou-lam2-nonsym-An} \\
        $5$ &
        $11$ &
        $11$ &
        $5$ &
        $5$ &
        $\PSL_{2}(11)$ &
        $\A_{5}$ &
        $\PSL_{2}(11)$ &
        \cite{a:ABD-PSL2,a:ABD-Un-CP,a:ABD-Un-CP-cor}\\
        $6$ &
        $28$ &
        $36$ &
        $9$ &
        $7$ &
        $\PSL_{2}(8)$ &
        $\D_{18}$ &
        $\PSL_{2}(8):3$ &
        \cite{a:ABD-Un-CP,a:ABD-Un-CP-cor} \\
        $7$ &
        $28$ &
        $252$ &
        $27$ &
        $3$ &
        $\PSU_{3}(3)$ &
        $3^{1+2}:8$ &
        $\PSU_{3}(3):2$ &
        \cite{a:ABD-Un-CP,a:ABD-Un-CP-cor} \\
        $8$ &
        $36$ & 
        $84$ &
        $14$ & 
        $6$ & 
        $\PSL_{2}(8)$ &
        $\D_{14}$ &
        $\PSL_{2}(8):3$ &
        \cite{a:ABDT-ASlam2,a:Montinaro-k2} \\
        $9$ &
        $126$ & 
        $1050$ & 
        $50$ & 
        $6$ & 
        $\PSU_{3}(5)$ &
        $5^{1+2}:8$ &
        $\PSU_{3}(5):2$ &
        \cite{a:ABDT-ASlam2} \\
        $10$ &
        $176$ & 
        $1100$ & 
        $50$ & 
        $8$ & 
        $\HS$ &
        $\PSU_{3}(5):2$ &
        $\HS$ &
        \cite{a:Zhou-lam2-nonsym} \\
        \noalign{\smallskip}\hline\noalign{\smallskip}
        
    \end{tabular}
\end{table}

In order to prove Theorem~\ref{thm:main} in Section~\ref{sec:proof},  we first observe that the group $G$ is point-primitive, and so the point-stabiliser $H$ is maximal in $G$. In particular,  flag-transitivity of $G$ implies that $H$ is large, that is to say, $|G|\leq |H|^{3}$, and then we can apply  \cite[Theorem~1.6]{a:ABD-Exp} and \cite{a:AB-Large-15} in which the large maximal subgroups of almost simple groups whose socle $X$ is a finite simple exceptional group of Lie type are determined. We then analyse all possible cases and prove that the only possible designs are those given in Theorem \ref{thm:main}.

\subsection{Definitions and notation}\label{sec:defn}

All groups and incidence structures in this paper are finite.
We here write $\Alt_{n}$ and $\Sym_{n}$ for the alternating group and the symmetric group on $n$ letters, respectively, and we denote by ``$[n]$'' a group of order $n$. Recall  that we in this paper adopt the standard Lie notation for groups of Lie type. For example, we sometimes write $A_{n-1}(q)$ and $A_{n-1}^{-}(q)$ in place of $\PSL_{n}(q)$ and $\PSU_{n}(q)$, respectively, $D_n^{-}(q)$ instead of  $\POm_{2n}^{-}(q)$, and $E_6^{-}(q)$ for ${}^2\!E_6(q)$.
We assume that $q>2$ when $G=G_2(q)$ since $G_{2}(2)$ is not simple and $G_2(2)' \cong A^{-}_{2}(3)$. Moreover, we view the Tits group ${}^2\!F_4(2)'$ as a sporadic group. We write $P_{i}$ to denote a standard maximal parabolic subgroup corresponding to deleting the $i$-th node in the Dynkin diagram of $X$, where we label the Dynkin diagram in the usual way, following \cite[p. 180]{b:KL-90}. We also use $P_{i,j}$ to denote the intersection of appropriate parabolic subgroups of
type $P_{i}$ and $P_{j}$.
For a given positive integer $n$ and a prime divisor $p$ of $n$, we denote the $p$-part of $n$ by $n_{p}$, that is to say, $n_{p}=p^{t}$ if $p^{t}\mid n$ but $p^{t+1}\nmid n$.
Further notation and definitions in both design theory and group theory are standard and can be found, for example, in \cite{b:Atlas,b:Dixon,b:KL-90,b:lander}.

\section{Preliminaries}\label{sec:pre}

In this section, we state some useful facts in both design theory and group theory.

\begin{lemma}\label{lem:param}
	Let $\Dmc$ be a $2$-design with parameters set $(v,k,\lambda)$. Then
	\begin{enumerate}[\rm (a)]
		\item $r(k-1)=\lambda(v-1)$;
		\item $vr=bk$;
		\item $v\leq b$ and $k\leq r$;
		\item $\lambda v<r^2$.
	\end{enumerate}
\end{lemma}
\begin{proof}
	Parts (a) and (b) follow immediately by simple counting. The inequality $v\leq b$ is the Fisher's inequality \cite[p. 57]{b:dembowski}, and so by applying part (b), we have that $k\leq r$. By part (a) and (c), we easily observe that  $r^{2}>r(k-1)=\lambda(v-1)>\lambda v$, and so $\lambda v<r^{2}$, as desired.
\end{proof}

Lemma \ref{lem:New} below is an elementary result on maximal subgroups of finite almost simple groups.

\begin{lemma}\label{lem:New}{\rm \cite[Lemma 2.2]{a:ABD-PSL2}}
	Let $G$  be an almost simple group with socle $X$, and let $H$ be maximal in $G$ not containing $X$. Then $G=HX$ and
	$|H|$ divides $|\Out(X)|\cdot |H\cap X|$.
\end{lemma}

\begin{lemma}\label{lem:Tits}
	Suppose that $\Dmc$ is block design with parameters $(v,k,\lambda)$ admitting a flag-transitive and point-primitive almost simple automorphism group $G$ with socle $X$ of Lie type in characteristic $p$. Suppose also that the point-stabiliser $G_{\alpha}$ does not contain $X$ and is not a parabolic subgroup of $G$. Then $\gcd(p,v-1)=1$.
\end{lemma}
\begin{proof}
	Note that $G_{\alpha}$ is maximal in $G$, then by Tits' Lemma \cite[1.6]{a:Seitz-TitsLemma}, $p$ divides $|G:G_{\alpha}|=v$, and so  $\gcd(p,v-1)=1$.
\end{proof}

If a group $G$ acts on a set $\Pmc$ and $\alpha\in \Pmc$, the \emph{subdegrees} of $G$ are the sizes of the orbits of the action of the point-stabiliser $G_\alpha$ on $\Pmc$. 

\begin{lemma}\label{lem:subdeg}{\rm \cite[3.9]{a:LSS1987}}
	If $X$ is a group of Lie type in characteristic $p$, acting on the set
	of cosets of a maximal parabolic subgroup, and $X$ is not $A_{n-1}(q)$, $D_{n}(q)$
	(with $n$ odd), or $E_{6}(q)$, then there is a unique subdegree which is a power of $p$.
\end{lemma}

\begin{remark}\label{rem:subdeg}
	We remark that even in the cases excluded in Lemma~\ref{lem:subdeg}, many of the maximal parabolic subgroups still have the property as asserted, see the  proof of  Lemma 2.6 in \cite{a:Saxl2002}. In particular, for an almost simple group $G$ with socle $X=E_{6}(q)$, if $G$  contains a
	graph automorphism or $H =P_{i}$ with $i=2$ or $4$, the conclusion of Lemma~\ref{lem:subdeg} is still true.
\end{remark}

\begin{lemma}\label{lem:six}
	Let $\Dmc$ be a $2$-design with parameters set $(v,k,\lambda)$, and let $\alpha$ be a point of $\Dmc$. If $G$ a flag-transitive automorphism group of $\Dmc$, then
	\begin{enumerate}[\rm (a)]
		\item $r\mid |G_{\alpha}|$;
		\item $r\mid \lambda d$, for all nontrivial subdegrees $d$ of $G$.
	\end{enumerate}
\end{lemma}
\begin{proof}
	Since $G$ is flag-transitive, the point stabiliser $G_{\alpha}$ is transitive on the  set of all blocks containing $\alpha$, and so $r=|G_{\alpha}:G_{\alpha,B}|$. Thus $r$ divides $|G_{\alpha}|$.  Part (b) is proved in \cite[p. 9]{a:Davies-87}.
\end{proof}

\begin{lemma}\label{lem:embed}{\rm \cite[2.2.5]{b:dembowski}} Let $\Dmc$ be a $2$-$(v,k,\lambda)$ design. If $\Dmc$ satisfies $r = k +\lambda$ and $\lambda\leq 2$, then $\Dmc$ embeds into a symmetric $2$-$(v + k +\lambda, k +\lambda,\lambda)$ design.
\end{lemma}

%
\begin{lemma}\label{lem:const}{\rm \cite[Proposition~4.1]{a:DLPB-2021-lam2}}
	Let  $\Smc = (\Pmc, \Lmc)$ be a $2$-$(v, k, 1)$ design with $k \geq 3$, and $\ell \in \Lmc$, and let $G\leq \Aut(\Smc)$. Let also
	$\Bmc = \{\ell \setminus  \{\alpha\} \mid \ell \in \Lmc,\ \alpha\in \ell\}$. Then  
	$\Dmc(\Smc) = (\Pmc, \Bmc)$  is a non-symmetric $2$-$(v, k-1, k-2)$
	design and $G$ is a subgroup of $\Aut(\Dmc(\Smc))$. Moreover, if $G$ is flag-transitive on $\Smc$ and $G_{\ell}$ is $2$-transitive on $\ell$, then $G$ is flag-transitive and point-primitive on $\Dmc(\Smc)$.
\end{lemma}

\section{Proof of the main result}\label{sec:proof}

Suppose that $\Dmc$ is a nontrivial $(v, k, \lambda)$ design admitting a flag-transitive almost simple automorphism group $G$ with socle $X$ being a finite simple exceptional group of Lie type. If $r$ is odd, then $r$ is coprime to $\lambda=2$, and so by \cite[Theorem 1.1]{a:A-Exp-CP}, we have that  $X={}^2\!B_{2}(q)$ or $X={}^2\!G_{2}(q)$ but in each of these cases $\lambda$ cannot be $2$ (see also \cite{a:ABFGMRTZ-CP}). Therefore, in what follows, we may assume that $r$ is even. 

Since $G$ is point-transitive, it follows from
Lemma~\ref{lem:New} and the point-stabiliser Theorem that
\begin{align}
	v=\frac{|X|}{|H\cap X|}.\label{eq:v}
\end{align}

For a point-stabiliser $H=G_{\alpha}$ of the  automorphism group $G$ of a flag-transitive design $\Dmc$, the subgroup $H$ is transitive on the set of blocks containing $\alpha$, and so $r$ divides $|H|$. On the other hand, the facts that $r(k-1)=\lambda(v-1)$ and $r>k$ imply that $\lambda v<r^{2}$. Thus  $\lambda|G|\leq |H|^{3}$, and hence the subgroup $H$ is \emph{large} in $G$, that is to say, $|H|^{3}\geq |G|$. We now apply \cite[Theorem 1.6]{a:ABD-Exp} and obtain the list of possibilities for the point-stabiliser subgroup  $H$ of $G$, and in conclusion, $H$ is either parabolic, or one of the subgroups recorded in \cite[Tables 2-3]{a:A-Exp-CP}.  For each such a possibility,  we note by Lemmas~\ref{lem:param}(a) and \ref{lem:six} that $r$ divides $|H|$, $2(v-1)$ and $2d$ for all nontrivial subdegrees $d$ of $G$. This enables us to find an upper bound $u_{r}$ for $r/2$. We  can moreover obtain a lower bound $\ell_{v}$ for $v$. The values $\ell_{v}$ and $u_{r}$ can be read of from \cite[Tables 2-3]{a:A-Exp-CP} with noting that in Table 3 of \cite{a:A-Exp-CP}, the values of $v$ and $r$ are listed for small $q$.

If $X$ and $H\cap X$  are as in Table 3 of \cite{a:A-Exp-CP}, then the possibilities of the parameters  $v$ and $r$ are as in the third and fourth columns of the same table, respectively. For each $v$, we know that $r$ divides $2(v-1)$, and so for each pair $(v,r)$, the parameter $b$ must divide $vr$. For each divisor $b$ of $vr$, we obtain the parameter $k$ by $k=vr/b$, and finally, for each $(v,b,r,k)$, we can then find $\lambda:=r(k-1)/(v-1)$ which has to be $2$. However, the possibilities recorded in \cite[Table 3]{a:A-Exp-CP}  give rise to no possible parameter set.

For each case recorded in Table 2 of \cite{a:A-Exp-CP}, we know a lower bound $\ell_{v}$ of $v$ and an upper bound $u_{r}$ of $r/2$. Thus, $r\leq 2 u_{r}$, and since $2v\leq r^{2}$ by Lemma~\ref{lem:param}(d), it follows that $\ell_{v}<2u_{r}^{2}$.
Running through all the possibilities recorded in \cite[Table 2]{a:A-Exp-CP}, we observe that $2u_{r}^2<\ell_{v}$ except for the case where $X=G_{2}(q)$ and $H\cap X=A_{2}^{\e}:2$. In this case $v=q^{3}(q^{3}+\e1)/2$ and $r$ divides $2(q^{3}-\e1)/\gcd(2,q-1)$. Then there exists a positive integer $m$ such that $rm\cdot \gcd(2,q-1)=2(q^{3}-\e1)$, and so the condition $\lambda v<r^{2}$ implies that
\begin{align*}
  q^{3}(q^{3}+\e1)<\frac{4(q^{3}-\e1)^{2}}{m^{2}\cdot \gcd(2,q-1)^{2}}.
\end{align*}
This inequality is true when $m=1$, or $(m,\e)=(2,-)$ and $q$ is even.
Let $q$ be odd and $(m,\e)=(1,+)$. Then $r=q^{3}-1$ and $k=q^{3}+3$, which contradicts Lemma~\ref{lem:param}(c). 
Let $q$ be odd and $(m,\e)=(1,-)$, or $q$ be even and $(m,\e)=(2,-)$. Then $v=q^{3}(q^{3}-1)/2$, $b=q^{3}(q^{3}+1)/2$, $r=q^{3}+1$ and $k=q^{3}-1$, and so $r=k+\lambda$ which is impossible by Lemma~\ref{lem:embed} and \cite{a:Regueiro-Exp}. 
Let now $q$ be even and $(m,\e)=(1,+)$. Then $v=q^{3}(q^{3}+1)/2$, $r=2(q^{3}-1)$ and $k=(q^{3}+4)/2$, and so Lemma~\ref{lem:param}(b) implies that $b=(2q^{9}-2q^{3})/(q^{3}+4)$. Since $2q^{9}-2q^{3}=(2q^6-8q^3+30)(q^{3}+4)-120$ and $b$ is a positive integer, it follows that $q^{3}+4$ is a divisor of $120$, which is correct if $q=2$, however, $G_{2}(2)$ is not simple, which is a contradiction. 
Let finally $q$ be even and $(m,\e)=(1,-)$. Then $v=q^{3}(q^{3}-1)/2$, $b=2(q^{6}-1)$, $r=2(q^{3}+1)$ and $k=q^{3}/2$ and $\lambda=2$, when $H_{0}:=H\cap X=\SU_{3}(q):2$ with $q$ even. Then  $H=H_{0}:T$ for some subgroup $T$ of $\Out(X)$.  
Let $B$ be a block of $\Dmc$ containing $\alpha$, and let  $K:=G_{B}$ and $L:=K\cap X$ and $S_{0}:=L\cap \SU_{3}(q)$. Then $S_{0}$ is a subgroup of $\SU_3(q)$ whose order is divisible by $q^3(q^2-1)/2$, and so by \cite[Table 8.5]{b:BHR-Max-Low}, we conclude that $S_{0}$ is (isomorphic) to a subgroup of a maximal parabolic subgroup $q^{1+2}:(q^2-1)$ of order $q^{3}(q^2-1)/2$ in $\SU_{3}(q)$, and hence $S_{0}$ has an element of order $q^{2}-1$. Since $S_0$ is a subgroup of $L\cap X$, we conclude that $K_{0}:=K\cap X$ has an element of order $q^{2}-1$. 
let now $N$ be a maximal subgroup of $G$ containing the block-stabiliser $K$. Then $|G:N|$ is divisible by $b=2(q^6-1)$, and so by inspecting the index of the maximal subgroups of $G$ from \cite{a:Cooperstein-G2-even}, we conclude that $N_0:=N\cap X$ is (isomorphic) to a parabolic subgroup $q^{1+4}:\GL_2(q)$ or $q^{2+3}:\GL_2(q)$. Thus 
$K_{0}=K\cap X$ is a subgroup of $N_{0}$ of order $q^{6}(q^{2}-1)/2$.  
Let $N_{0}=Q:M$, where $Q=[q^5]$ and $M=\GL_{2}(q)$. Then 
the intersection between $K_{0}\cap Q$ is a normal subgroup of $K_{0}$ of order $q^5$ or $q^5/2$, and so $K_{0}Q/Q$ is isomorphic to a subgroup $W$ of $M$ of order $q(q^2-1)$ or $q(q^2-1)/2$, respectively. Recall that $K_{0}$ has an element of order $q^{2}-1$. Therefore,  $W$ has a cyclic group of order $q^2-1$, which in particular implies that $W$ contains the center $Z$ of $M=\GL_2(q)$, and so $W/Z$ is a subgroup of $\PGL_{2}(q)$ containing a cyclic subgroup of order $q+1$. Now by inspecting the subgroups of $\PGL_{2}(q)$, we easily conclude that this is possible only when $q=2$ or $4$. Since $G_{2}(2)$ is not simple, we have that $X=G_{2}(4)$, and so $G=G_{2}(4)$ or $G_{2}(4):2$, and $(v,b,r,k,\lambda)=(2016, 8190, 130, 32, 2 )$. Assume first $G=G_{2}(4)$. Then by GAP \cite{GAP4}, we observe that $G$ has two nonconjugate subgroups $K_1$ and $K_{2}$ of index $b=8190$. We also observe that $K_{1}$ has orbit lengths $16$, $80$ and $1920$, and $K_{2}$ has orbit lengths $480$ and $1536$. 
But $G$ is flag-transitive, and so $k$ must be the length of a $K_{1}$-orbit or $K_{2}$-orbit on $v=2016$, which is a contradiction. By the same argument, the  case where $G=G_{2}(4):2$ can be ruled out.

It remains to consider the case where $H$ is a parabolic subgroup of $G$. In what follows, we further assume that $K=G_B$ and $K_{0}=K\cap X$, where $B$ is a block containing $\alpha$. We now continue our argument by case by case analysis. We note here that the value of the  parameter $v$ in each case, can be read off from~\cite[Table 4]{a:ABD-Exp}.

Suppose  that $X={}^2\!B_{2}(q)$ and $H\cap X=q^{2}{:}(q-1)$ with $q=2^a$ for some odd $a\geq 3$. Then by \eqref{eq:v}, we have that  $v=q^2+1$, and so $r(k-1)=2(v-1)=2q^2=2^{2a+1}$. Let $r=2^{c}$. Then $k=2^{2a-c+1}+1$. 
Note that $b=rv/k$, $v=q^2+1$ and $G$ is transitive on the set of blocks of $\Dmc$. Then $|G:K|=b=2^{c}(q^2+1)/k$, where $K=G_{B}$ with $\alpha \in B$. Assume now that  $N_{0}$ is a maximal subgroup of $X$ containing $K_{0}=K\cap X$. Then $|X:N_{0}|$ must divide $b$. The knowledge of maximal subgroups of $X={}^2\!B_{2}(q)$ shows that $K_{0}$ embeds into a parabolic subgroup $N_{0}\cong q^{2}{:}(q-1)$ of index $q^{2}+1$. Moreover, $|X:K_{0}|$ divides $b=2^{c}(q^2+1)/k$, and since $|X:N_{0}|=q^{2}+1$ divides $|X:K_{0}|$, it follows that $k$ divides $2^{c}$, which is impossible as $k$ is odd.

Suppose that $X={}^2\!G_{2}(q)$ and $H\cap X=q^3{:}(q-1)$ with $q=p^{a}=3^{2m+1}\geq 27$. By \eqref{eq:v}, we have that  $v=q^3+1$.  Since $r$ is even, it follows from Lemma~\ref{lem:param}(a) that $r$ divides $\lambda(v-1)=2q^{3}$, and so $r=2\cdot 3^{c}$ for some $1\leq c\leq 3a$. Again Lemma~\ref{lem:param}(a) implies that $k=3^{t}+1$, where $t:=3a-c$. We first observe that $t\geq 1$ as $k\geq 3$. 
Since $v=q^3+1$ and $k=3^{t}+1$, it follows that $|G:K|=b=vr/k=2\cdot 3^{c}\cdot (q^3+1)/(3^{t}+1)$, where $K=G_{B}$ with $\alpha\in B$. Assume now that $N$ is a maximal subgroup of $X$ containing $K$. Then $|G:N|$ divides $b$. The list of maximal subgroups of $X=\!^{2}\!G_{2}(q)$ can be read off from \cite[Table 8.43]{b:BHR-Max-Low} (see also \cite{a:Levchuk-Ree}), and so by inspecting the index of the maximal subgroups $N$ in $G$, we conclude that $K$ is contained in $N=(\langle z \rangle\times A_{1}(q)):T$, where $T\leq\Out(X)$ and $z$ is an involution in $X$. In this case, as $|G:N|=q^{2}(q^{2}-q+1)$ divides $b=2\cdot 3^{c}\cdot (q^3+1)/k$, it follows that $k=3^{t}+1$ divides $2(q+1)$. Then  $t\leq a$, or equivalently, $c\geq 2a$. Moreover, if  $N_{0}=N\cap X$, then $N_{0}=\langle z \rangle\times A_{1}(q)$ is a maximal subgroup of $X$ containing $K_{0}=K\cap X$.  
Let $\beta$ and $\gamma$ be two distinct points in $B\setminus\{\alpha\}$. Then $X_{\alpha,\beta}$ is isomorphic to a cyclic group of order $q-1$ which has index $q^{3}$ in $H_{0}$, see \cite{a:Tits-61} or \cite[Lemma 3.2]{a:Pierro-Ree-16}. Moreover, $X_{\alpha,\beta,\gamma}$ is generated by an involution, and hence the nontrivial $X_{\alpha,\beta}$-orbits are of length $q-1$ or $(q-1)/2$ depending on whether $\gamma$ is fixed by $X_{\alpha,\beta}$ or not, respectively. Let now $C\neq B$ be the block of $\Dmc$ passing through both $\alpha$ and $\beta$. Then $B\cup C$ is fixed by $X_{\alpha,\beta}$, and so $J:=(B\cup C)\setminus\{\alpha,\beta\}$ is a union of nontrivial $X_{\alpha,\beta}$-orbits. 
Thus $(q-1)/2$ divides $|J|=2(k-1)-|B\cap C|=2\cdot 3^{t}-|B\cap C|$. As $|B\cap C|\geq 2$, we have that  $|J|\leq 2\cdot3^{t}-2$, and so $(q-1)/2\leq 2\cdot3^{t}-2$. Thus $3^{a}+3\leq 4\cdot 3^{t}$. Since $t\leq a$, this inequality holds only for $t=a-1$ or $a$, and so $k=(q+3)/3$ or $q+1$, respectively. 
If  $k=(q+3)/3$, then $b=18q^{2}(q^{3}+1)/(q+3)$, and so $(q+3)/3=3^{a-1}+1$ must divide $2(3^{3a}+1)$. Since $\gcd(3^{a-1}+1,3^{3a}+1)=2$, it follows that $3^{a-1}+1$ divides $4$, and so $a=1$ or $2$, which is impossible. Thus, $k=q+1$, and hence $r=2q^{2}$ and $b=2q^{2}(q^{2}-q+1)$. This also requires $K_{0}=A_{1}(q)$ and $K=A_{1}(q):T$. Therefore, $N_{0}\cap H_{0}=\langle z \rangle\times (q:\frac{q-1}{2})$  by \cite[Lemma~3.2]{a:Pierro-Ree-16}, and so  $L_{0}:=X_{\alpha,B}=H_{0}\cap K_{0}=q:\frac{q-1}{2}$ and $L:=H\cap K=(q:\frac{q-1}{2}):T$. Note that $K_{0}$ acts $2$-transitively on $B$. Then $X$ is also flag-transitive. Note also that $N_{0}=C_{X}(z)$ is the line-stabiliser $X_{\ell}$ of the Ree Unital $\Umc_{R}(q)$, where the line $\ell$ contains all $q+1$ fixed points of the involution $z$. Since $K_{0}=X_{B}$ is a normal subgroup of $X_{\ell}=N_{0}$, it follows that $\ell$ is a union of $X_{B}$-orbits. Since $|\ell|=|B|=q+1$ and $z$ fixes $\alpha\in B$ (in fact it fixes at least two points in $B$), we conclude that $B=\ell$, and so $z$ fixes $B$, that is to say, $z\in X_{B}=K_{0}$, which is a contradiction.  

\begin{table}
	\centering
	\scriptsize
	\caption{Some parameters for some parabolic subgroups of almost simple groups with socle $X$.}\label{tbl:parab}
		\begin{tabular}{llllll}
			\hline\noalign{\smallskip}
			\multicolumn{1}{l}{Line} &
			\multicolumn{1}{l}{$X$} &
			\multicolumn{1}{l}{$H\cap X$} &
			\multicolumn{1}{l}{$v$} &
			\multicolumn{1}{l}{$|v-1|_{p}$} &
			\multicolumn{1}{l}{Comments}
			\\
			\noalign{\smallskip}\hline\noalign{\smallskip}
			$1$  &
			$^{3}\!D_{4}(q)$&
			$P_{1}$  &
			$\Phi_{2}\Phi_{3}\Phi_{6}\Phi_{12}$  &
			$q^{3}$ &
			\\
			$2$ &
			$^{2}\!F_{4}(q)$&
			$P_{1}$  &
			$\Phi_{2}\Phi_{4}^{2}\Phi_{6}\Phi_{12}$ &
			$q^{2}$&
			\\
			$3$ &
			$F_{4}(q)$&
			$P_{1,4}$  &
			$\Phi_{2}^{2}\Phi_{3}^{2}\Phi_{4}\Phi_{6}^{2}\Phi_{8}\Phi_{12}$ &
			$2q$ &
			$q=2^{a}$, $H$ contains graph automorphism \\
			$4$ &
			$F_{4}(q)$&
			$P_{2,3}$  &
			$\Phi_{2}^{2}\Phi_{3}^{2}\Phi_{4}^{2}\Phi_{6}^{2}\Phi_{8}\Phi_{12}$ &
			$2q$ &
			$q=2^{a}$, $H$ contains graph automorphism \\
			%
			%
			%
			$5$  &
			$E_{6}(q)$   &
			$P_{1,6}$ &
			$\Phi_{3}^{2}\Phi_{5}\Phi_{6}\Phi_{8}\Phi_{9}\Phi_{12}$&
			$\gcd(2,p)q$ &
			$H$ contains graph automorphism \\
			$6$  &
			$E_{6}(q)$   &
			$P_{3,5}$ &
			$\Phi_{2}\Phi_{3}^{2}\Phi_{4}^{2}\Phi_{5}\Phi_{6}^{2}\Phi_{8}\Phi_{9}\Phi_{12}$&
			$\gcd(2,p)q$ &
			$H$ contains graph automorphism \\
			\noalign{\smallskip}\hline
			Note: & \multicolumn{5}{p{13cm}}{$\Phi_{n}:=\Phi_{n}(q)$ is the $n$-th cyclotomic polynomial with $q=p^{a}$ and $p$ prime.}
		\end{tabular}
\end{table}

Suppose that $X=E_{6}(q)$ and $H\cap X=P_{1}$. Then 
$H\cap X=[q^{16}]{:}D_{5}(q)\cdot (q-1)$. Then $v=(q^8+q^4+1)(q^9-1)/(q-1)$. Note by \cite{a:LSS-rank3} that $G$ has nontrivial subdegrees  $q(q^{8}-1)(q^{3}+1)/(q-1)$ and $q^{8}(q^{5}-1)(q^{4}+1)/(q-1)$, and so by Lemma~\ref{lem:six}(b), we conclude that $r$ divides $q(q^{4}+1)$, and so Lemma \ref{lem:param}(d) implies that $v<r^2<q^{12}$, which is a contradiction.

Suppose that $X=E_{6}(q)$ and  $H\cap X=P_{3}$. Then $H\cap X=[q^{25}]{:}A_{1}(q)A_{4}(q){\cdot}(q-1)$. Then $v=(q^{3}+1)(q^{4}+1)(q^{9}-1)(q^{12}-1)/(q-1)(q^{2}-1)$. It follows from \cite{a:Korableva-E6E7} that $X$ has subdegrees $q(q^5-1)(q^4-1)/(q-1)^2$ and $q^{13}(q^{5}-1)/(q-1)$. Moreover, by Lemmas \ref{lem:param} and \ref{lem:six}, the parameter $r$ is a divisor of  $12aq(q^{5}-1)/(q-1)$, but then $2v>r^{2}$, which is a contradiction.

We finally consider the remaining cases. For each remaining $(X,H\cap X)$, the $p$-part $|v-1|_p$ is $q$ except for those listed in Table~\ref{tbl:parab} for which $|v-1|_p$ is recorded in the  fifth column of the same table. Note by Lemma~\ref{lem:subdeg} and Remark~\ref{rem:subdeg} that $G$ has a unique prime power subdegree $p^{n}$. Then by Lemmas \ref{lem:param} and \ref{lem:six}, $r$ must divide $2\gcd( v-1,p^{n})$ implying that $r$ is a divisor of $2u_{r}$ where $u_{r}:=|v-1|_p$. Since $2v<r^2$, we conclude that $v<2(|v-1|_p)^2$, however, this inequality does not hold for all the remaining cases. For example, if $X={}^3\!D_{4}(q)$ and $H\cap X=P_{1}$. Then $H\cap X=[q^{11}]{:}\SL_{2}(q) {\cdot}(q^{3}-1)$, and so $v=(q^8+q^4+1)(q^3+1)$ and $r$ divides $2|v-1|_p=2q^3$, and hence  $q^{11}<v<2r^2<2q^6$, which is a contradiction.

\section*{Acknowledgements}

The author is grateful to Alice Devillers and Cheryl E. Praeger for supporting his visit to UWA (The University of Western Australia) during February–June 2023. He also thanks Bu-Ali Sina University for the support during his sabbatical leave. The author would like to thank Alessandro Montinaro for the comments on the proof of the main result. 

\section{Declaration of competing interest}

The author confirms that this manuscript has not been published elsewhere. It is not also under consideration by another journal. He also confirms that there are no known conflicts of interest associated with this publication. He has no competing interests to declare that are relevant to the content of this article, and he confirms that availability of data and material is not applicable. The author declares that he has no known competing financial interests or personal relationships that could have appeared to influence the work reported in this paper.

\bibliographystyle{elsart-num-sort}


\end{document}